\documentclass[11pt]{amsart}
\usepackage[utf8]{inputenc}

\usepackage{enumitem}
\usepackage{amsfonts}
\usepackage{amsmath}
\usepackage{amsthm}
\usepackage{amssymb}
\usepackage{mathrsfs}
\usepackage{enumitem}
\usepackage{forest}
 
\usepackage[all]{xy}
\usepackage{appendix}

\usepackage{tikz}
\usetikzlibrary{shapes.geometric}
\usetikzlibrary{positioning}
\usetikzlibrary{automata}
\usetikzlibrary{decorations}
\usetikzlibrary{decorations.pathreplacing}

\usepackage{tikz-cd}

\theoremstyle{definition}
\newtheorem{defi}{Definition}[section]

\newtheorem{rem}[defi]{Remark}
\newtheorem{prop}[defi]{Proposition}
\newtheorem{cor}[defi]{Corollary}
\newtheorem{lem}[defi]{Lemma}
\newtheorem{theo}[defi]{Theorem}
\newtheorem{con}[defi]{Conjecture}
\newtheorem{fact}[defi]{Fact}

\title{Freehedra are short}

\author{Daria Poliakova}

\newcommand{\F}{\mathcal{F}}

\begin{document}
\maketitle
 \begin{abstract}
We prove the combinatorial property of shortness for freehedra. Note that associahedra, a sibling family of polytopes, are not short.
\end{abstract}

\section*{Introduction}
In this paper we prove that freehedra of \cite{San} (see also \cite{RS}) are short, which confirms the hope that their cellular complexes are integrated $A_\infty$-coalgebras as in \cite{AP}. Compare: in Section 6.2 of \cite{San} Samson Saneblidze states that an explicit $A_\infty$-extension of his cellular diagonal is expected. \\

Let us provide some context. In \cite{ACD}, tensor products for representations up to homotopy and for their morphisms were constructed, but for morphisms, those tensor products were neither strictly associative nor consistent with compositions. The description of the full coherent structure was left as an open question. In \cite{Pol} the author has demonstrated that questions about tensoring morphisms of representations up to homotopy can be reformulated in terms of polytopal diagonals, and the corresponding polytopes are the freehedra. In \cite{AP}, for any directed polytope a colored operad is constructed, with colors enumerated by all faces of the polytope. For sufficiently nice (namely, {\em short}) polytopes, the Poincar\'e-Hilbert series encodes a family of operations that conjecturally assemble into the structure of an {\em integrated $A_\infty$-coalgebra} (over $\mathbb{F}_2$, i.e. signs are not established). There is a forgetful functor from integrated $A_\infty$-coalgebras to usual $A_\infty$-coalgebras. \\

The main drama of \cite{AP} was that associahedra failed to be short, and it was not obvious that operadically meaningful short polytopes with non-coassociative diagonals exist at all. The present paper provides an example.\\

We now outline the structure of the paper. In Section 1, we remind the construction of freehedra. In Section 2, we explain the construction of an operad associated to a directed polytope, and explain shortness. In Section 3, we prove that freehedra are short.\\

\subsection*{Acknowledgements} We are grateful to Sergey Arkhipov for supporting the current work, and to Samson Saneblidze for previously teaching us some combinatorics of freehedra.
\section{Freehedra}

Freehedra (aka Hochschild polytopes) were introduced by Saneblidze as subdivisions of cubes. In \cite{Pol}, the author has shown that forest-tree-forest triples of \cite{ACD} provide a set of labels for the faces of freehedra, consistent with their cellular structure. In this section we remind both approaches, with specific attention to vertices.\\

\subsection{Subdivisions of cubes and coordinates}

Freehedra are defined as subdivisions of cubes inductively, where at each step the newly constructed freehedron $\F^n$ comes with a distinguished facet $X_n$. 

\begin{defi}
\label{free}
We set $\mathcal{F}^0$ to be the point, and we set $\mathcal{F}^1$ to be the interval $[0,2]$ with distinguished vertex $X^1$ = $2$. Now assume $\mathcal{F}^{n-1}$ and its distinguished facet $X^{n-1}$ are defined. Consider $F^{n-1} \times [0,2]$, and split its facet $X^{n-1} \times [0,2]$ vertically into $X^{n-1} \times [0,1]$ and $X^{n-1} \times [1,2]$. This is $\mathcal{F}^n$. Set $X^n = X^{n-1} \times [1,2]$.
\end{defi}

The figure below features $\F^1$, $\F^2$ and $\F^3$, with their distinguished facets highlighted red.

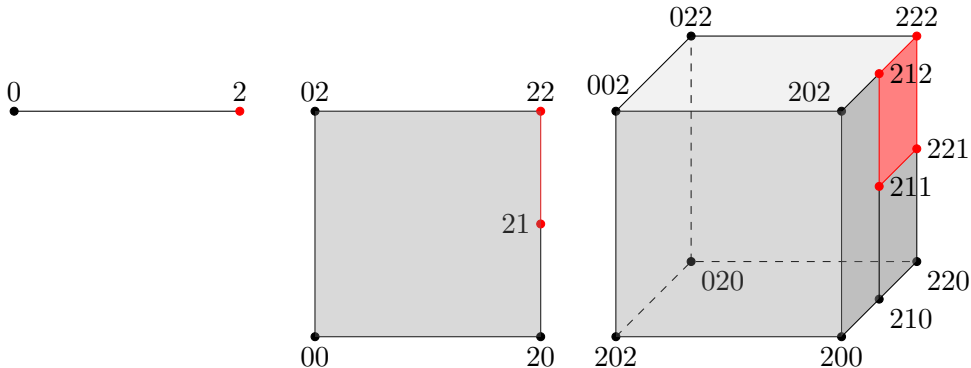
\begin{center}
\begin{tikzpicture}
\draw[thin] (0,0) -- (3,0);
\filldraw[black] (0,0) circle (1.5pt);
\node[anchor = south] at (0,0) {0};
\filldraw[red] (3,0) circle (1.5pt);
\node[anchor = south] at (3,0) {2};

\draw[thin] (3+4,0) -- (0+4,0) -- (0+4,-3) -- (3+4,-3) -- (3+4,-1.5);
\draw[red] (3+4,-1.5) -- (3+4,-0);
\filldraw[black] (0+4,0) circle (1.5pt);
\node[anchor = south] at (4,0) {02};

\filldraw[black] (0+4,-3) circle (1.5pt);
\node[anchor = north] at (4,-3) {00};

\filldraw[red] (3+4,0) circle (1.5pt);
\node[anchor = south] at (7,0) {22};

\filldraw[red] (3+4,-1.5) circle (1.5pt);
\node[anchor = east] at (7,-1.5) {21};

\filldraw[black] (3+4,-3) circle (1.5pt);
\node[anchor = north] at (7,-3) {20};

\fill[gray, opacity = 0.3] (0+4,0) -- (3+4,0) -- (3+4,-3) -- (0+4,-3) -- cycle;

\draw[thin] (8,0) rectangle (11,-3);
\draw[dashed] (8,-3) -- (9,-2) -- (12,-2);
\draw[dashed] (9,-2) -- (9,1);
\filldraw[black] (8,0) circle (1.5pt);
\node[anchor = south] at (7.9,0.1) {002};

\filldraw[black] (8,-3) circle (1.5pt);
\node[anchor = north] at (8,-3) {202};

\filldraw[black] (11,0) circle (1.5pt);
\node[anchor = south east] at (11,0) {202};

\filldraw[black] (11,-3) circle (1.5pt);
\node[anchor = north] at (11,-3) {200};

\filldraw[black] (9,1) circle (1.5pt);
\node[anchor = south] at (9,1) {022};

\filldraw[black] (9,-2) circle (1.5pt);
\node[anchor = north west] at (9,-2) {020};

\filldraw[black] (11.5,-2.5) circle (1.5pt);
\node[anchor = north west] at (11.5,-2.5) {210};

\filldraw[black] (12,-2) circle (1.5pt);
\node[anchor = north west] at (12,-2) {220};

\fill[gray,opacity = 0.3] (8,0) rectangle (11,-3);
\fill[gray,opacity = 0.1] (8,0) -- (9,1) -- (12,1) -- (11,0);
\fill[gray,opacity = 0.5] (11,-3) -- (11,0) -- (11.5,0.5) -- (11.5,-2.5) -- cycle;
\fill[gray,opacity = 0.5] (11.5,-2.5) -- (11.5,-1) -- (12,-0.5) -- (12,-2) -- cycle;

\draw[thin] (8,0) -- (9,1) -- (12,1);
\draw[thin] (11,0) -- (11.5,0.5);
\draw[red] (11.5,0.5) -- (12,1) -- (12,1-1.5) --(11.5,-1) -- cycle;
\fill[red, opacity = 0.5] (11.5,0.5) -- (12,1) -- (12,1-1.5) --(11.5,-1) -- cycle;
\draw[thin] (11,-3) -- (11.5,-2.5) -- (11.5,-1);
\draw[thin] (11.5,-2.5) -- (12,-2) -- (12,-0.5);

\filldraw[red] (12,1) circle (1.5pt);
\node[anchor = south] at (12,1) {222};

\filldraw[red] (12,-0.5) circle (1.5pt);
\node[anchor = west] at (12,-0.5) {221};

\filldraw[red] (11.5,0.5) circle (1.5pt);
\node[anchor = west] at (11.5,0.5) {212};

\filldraw[red] (11.5,-1) circle (1.5pt);
\node[anchor = west] at (11.5,-1) {211};

\end{tikzpicture}
\end{center}

\begin{rem}
This is Definition 4.2 from \cite{Pol}, scaled by 2 for the simplicity of notation. Saneblidze does not use this scaling in [San].
\end{rem}

\begin{prop}
Vertices of $\F^n$ correspond to all coordinate $n$-tuples $v$ with entries are $0$, $1$ or $2$, subject to condition $\bigstar$: $v_i = 1$ implies $i>1$ and $v_j \neq 0$ for $j < i$.
\end{prop}

\begin{proof}
We first show by induction that vertices in distinguished facets do not have $0$ coordinates. This holds for $X^1 = 2$. Now assume that this is known for $X^{n-1}$. By by Definition \ref{free}, the coordinates $n$-tuples for vertices in $X^n$ are of the form $(v,1)$ and $(v,2)$ where $v$ is a coordinate $n$-tuple for the vertex of $X^n$. By inductive assumption $v$ does not have zero entries, so neither do  $(v,1)$ and $(v,2)$.\\

We now show by induction that all vertices of freehedra have coordinate $n$-tuples satisfying $\bigstar$. This holds for $\F^1$. Now assume that this is known for $\F^{n-1}$. Vertices of $\F^n$ are vertices of $\F^{n-1}$ with $0$, $1$ or $2$ appended in the end, and $1$ can only be appended if the respective vertex of $\F^{n-1}$ was inside $X^{n-1}$. By the argument above, this vertex did not have any $0$ coordinates, so the newly formed vertex also does not have them. And certainly appending a number does not change the $v_0$ which is not allowed to be 1. \\

Finally, to show that any such $n$-tuple corresponds to a vertex of freehedra, we inductively show that there are $2^{n-2}(n+3)$ such $n$-tuples, which is the number of vertices in $\F^n$ by Proposition 8.19 in (Chapoton). The base is the case of $n=1$: indeed, $2^{-1}(1+3)=2$. Now assume that for $n$ the result is proved. Note that any $\bigstar$-sequence of length $n+1$ can be obtained from a $\bigstar$-sequence of length $n$ by appending a number. If this shorter sequence has no $0$ entries, this number can be $0$, $1$ or $2$; otherwise it can be only $0$ or $2$. Now the number of $\bigstar$-sequences of length $n$ with no $0$ entries is $2^{n-1}$, so, by inductive assumption, the number of other $\bigstar$-sequences is $2^{n-2}(n+3)-2^{n-1} = 2^{n-2}(n+1)$. Then the number of $\bigstar$-sequences of length $n+1$ is

$$ 2^{n-1} \times 3 + 2^{n-2}(n-1) \times 2 = 2^{(n+1)-2}((n+1)+3)$$
\end{proof}

\begin{rem}
In [Cha] and [Com], freehedra (under the name of Hochshild polytopes) appear in a slightly different realization where vertices also correspond to words in the alphabet \{0,1,2\}. The difference is seen by comparing Figure 2.1 in [Com] to our figure above.
\end{rem}

\subsection{Forest-tree-forest triples}

We now remind forest-tree-forest triples of \cite{ACD}.

\begin{defi}
A short forest is a sequence of nonempty planar trees of depth 2. Inner edges are called branches and outer edges are called leaves.
\end{defi}
\begin{defi}
A forest-tree-forest triple is the data of one left short forest (possibly empty), one middle lone standing tree (possibly empty), and one right short forest (possibly empty).
\end{defi}

Below is en example of such a triple, where left forest has two trees and right forest is empty.

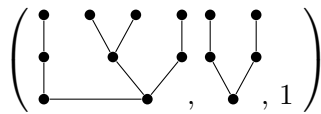
\[ \left ( \vcenter{\hbox{ \begin{forest}
for tree = {grow'=90,circle, fill, minimum width = 4pt, inner sep = 0pt, l sep = 0, s sep = 13pt, l-=5mm}
[{},phantom
[{},name = one [[]] ]
[{},name = two [[][]] [[]] ]
]
\draw[black] (one) -- (two);
\end{forest},
\begin{forest}
for tree = {grow'=90,circle, fill, minimum width = 4pt, inner sep = 0pt, s sep = 13pt, l sep = 0, l-=5mm}
[{},phantom
[ [[]] [[]]  ]
]
\end{forest}, 1  }} \right ) \]

By a {\em mid-branch space} we mean a couple of neighbouring branches that belong to the same tree. To compute the dimension of a forest-tree-forest triple, we count the number of mid-branch spaces, and add 1 if the middle tree is nonempty. In the example above, there is one mid-branch space in the second tree of the left forest and one mid-branch in the middle tree which is nonempty, so the dimension is 2+1 = 3.\\

\begin{defi} The {\em face transformations} for forest-tree-forest triples are the following:

\begin{itemize}
    \item {\sc (Merge) }for a mid-branch space in any tree (including the middle one), replace the two branches by one branch having the leaves of both:
    
    \[ \left ( \vcenter
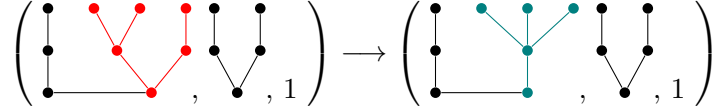
{\hbox{ \begin{forest}
for tree = {grow'=90,circle, fill, minimum width = 4pt, inner sep = 0pt, l sep = 0, s sep = 13pt, l-=5mm}
[{},phantom
[{},name = one [[]] ]
[{},name = two, for tree = {fill = red, edge = {color = red}} [[][]] [[]] ]
]
\draw[black] (one) -- (two);
\end{forest},
\begin{forest}
for tree = {grow'=90,circle, fill, minimum width = 4pt, inner sep = 0pt, s sep = 13pt, l sep = 0, l-=5mm}
[{},phantom
[ [[]] [[]]  ]
]
\end{forest}, 1  }} \right )  \longrightarrow 
\left ( \vcenter{\hbox{ \begin{forest}
for tree = {grow'=90,circle, fill, minimum width = 4pt, inner sep = 0pt, l sep = 0, s sep = 13pt, l-=5mm}
[{},phantom
[{},name = one [[]] ]
[{},name = two, for tree = {fill = teal, edge = {color = teal}}  [[][][]] ]
]
\draw[black] (one) -- (two);
\end{forest},
\begin{forest}
for tree = {grow'=90,circle, fill, minimum width = 4pt, inner sep = 0pt, s sep = 13pt, l sep = 0, l-=5mm}
[{},phantom
[ [[]] [[]]  ]
]
\end{forest}, 1  }} \right )
\] 
    
    \item {\sc (Push apart) }for a mid-branch space in any non-middle tree, split the tree in two:
    
    \[ \left ( \vcenter
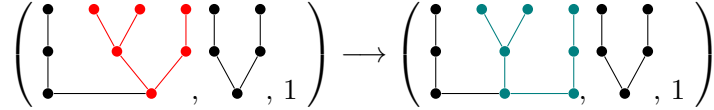
{\hbox{ \begin{forest}
for tree = {grow'=90,circle, fill, minimum width = 4pt, inner sep = 0pt, l sep = 0, s sep = 13pt, l-=5mm}
[{},phantom
[{},name = one [[]] ]
[{},name = two, for tree = {fill = red, edge = {color = red}} [[][]] [[]] ]
]
\draw[black] (one) -- (two);
\end{forest},
\begin{forest}
for tree = {grow'=90,circle, fill, minimum width = 4pt, inner sep = 0pt, s sep = 13pt, l sep = 0, l-=5mm}
[{},phantom
[ [[]] [[]]  ]
]
\end{forest}, 1  }} \right )  \longrightarrow 
\left ( \vcenter{\hbox{ \begin{forest}
for tree = {grow'=90,circle, fill, minimum width = 4pt, inner sep = 0pt, l sep = 0, s sep = 13pt, l-=5mm}
[{},phantom
[{},name = one [[]] ]
[{},name = two, for tree = {fill = teal, edge = {color = teal}}  [[][]] ]
[{},name = three, for tree = {fill = teal, edge = {color = teal}}  [[]] ]
]
\draw[black] (one) -- (two);
\draw[teal] (two) -- (three);
\end{forest},
\begin{forest}
for tree = {grow'=90,circle, fill, minimum width = 4pt, inner sep = 0pt, s sep = 13pt, l sep = 0, l-=5mm}
[{},phantom
[ [[]] [[]]  ]
]
\end{forest}, 1  }} \right )
\] 

    \item {\sc (Move left) } for any number of branches in the middle tree counting from the left, remove them from the middle tree and add to the left forest as a separate rightmost tree:
    
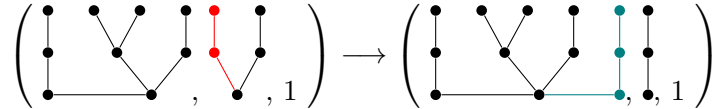
\[ \left ( \vcenter{\hbox{ \begin{forest}
for tree = {grow'=90,circle, fill, minimum width = 4pt, inner sep = 0pt, l sep = 0, s sep = 13pt, l-=5mm}
[{},phantom
[{},name = one [[]] ]
[{},name = two [[][]] [[]] ]
]
\draw[black] (one) -- (two);
\end{forest},
\begin{forest}
for tree = {grow'=90,circle, fill, minimum width = 4pt, inner sep = 0pt, s sep = 13pt, l sep = 0, l-=5mm}
[{},phantom
[ [{},  for tree = {fill = red, edge = {color = red}} []] [[]]  ]
]
\end{forest}, 1  }} \right )  \longrightarrow 
\left ( \vcenter{\hbox{ \begin{forest}
for tree = {grow'=90,circle, fill, minimum width = 4pt, inner sep = 0pt, l sep = 0, s sep = 13pt, l-=5mm}
[{},phantom
[{},name = one [[]] ]
[{},name = two  [[][]][[]] ]
[{},name = three, for tree = {fill = teal, edge = {color = teal}}  [[]] ]
]
\draw[black] (one) -- (two);
\draw[teal] (two) -- (three);
\end{forest},
\begin{forest}
for tree = {grow'=90,circle, fill, minimum width = 4pt, inner sep = 0pt, s sep = 13pt, l sep = 0, l-=5mm}
[{},phantom
[ [[]]  ]
]
\end{forest}, 1  }} \right )
\] 
    
    \item  {\sc (Move right) } for any number of branches in the middle tree counting from the right, remove them from the middle tree and add to the right forest as a separate leftmost tree:
    
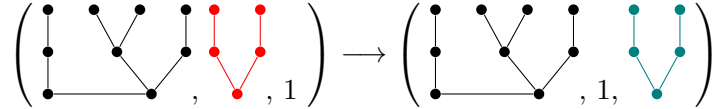
\[ \left ( \vcenter{\hbox{ \begin{forest}
for tree = {grow'=90,circle, fill, minimum width = 4pt, inner sep = 0pt, l sep = 0, s sep = 13pt, l-=5mm}
[{},phantom
[{},name = one [[]] ]
[{},name = two [[][]] [[]] ]
]
\draw[black] (one) -- (two);
\end{forest},
\begin{forest}
for tree = {grow'=90,circle, fill, minimum width = 4pt, inner sep = 0pt, s sep = 13pt, l sep = 0, l-=5mm}
[{},phantom
[{},  for tree = {fill = red, edge = {color = red}} [ []] [[]]  ]
]
\end{forest}, 1  }} \right )  \longrightarrow 
\left ( \vcenter{\hbox{ \begin{forest}
for tree = {grow'=90,circle, fill, minimum width = 4pt, inner sep = 0pt, l sep = 0, s sep = 13pt, l-=5mm}
[{},phantom
[{},name = one [[]] ]
[{},name = two [[][]] [[]] ]
]
\draw[black] (one) -- (two);
\end{forest},
 1, \begin{forest}
for tree = {grow'=90,circle, fill, minimum width = 4pt, inner sep = 0pt, s sep = 13pt, l sep = 0, l-=5mm}
[{},phantom
[{},  for tree = {fill = teal, edge = {color = teal}} [ []] [[]]  ]
]
\end{forest}  }} \right ) 
\] 
    
\end{itemize}
\end{defi}

Tree-forest-tree triples were shown in \cite{Pol} to be in a bijective correspondence with faces of freehedra, thus labelling them. Moreover, the face labelled $(\mathfrak{F},\mathfrak{T},\mathfrak{G})$ is included in the face labelled $(\mathfrak{F}',\mathfrak{T}',\mathfrak{G}')$ if and only if $(\mathfrak{F},\mathfrak{T},\mathfrak{G})$ can be obtained from $(\mathfrak{F}',\mathfrak{T}',\mathfrak{G}')$ by a sequence of face transformation described above. The picture below illustrates labelling for $\F^2$, with face transformations indicated. \\

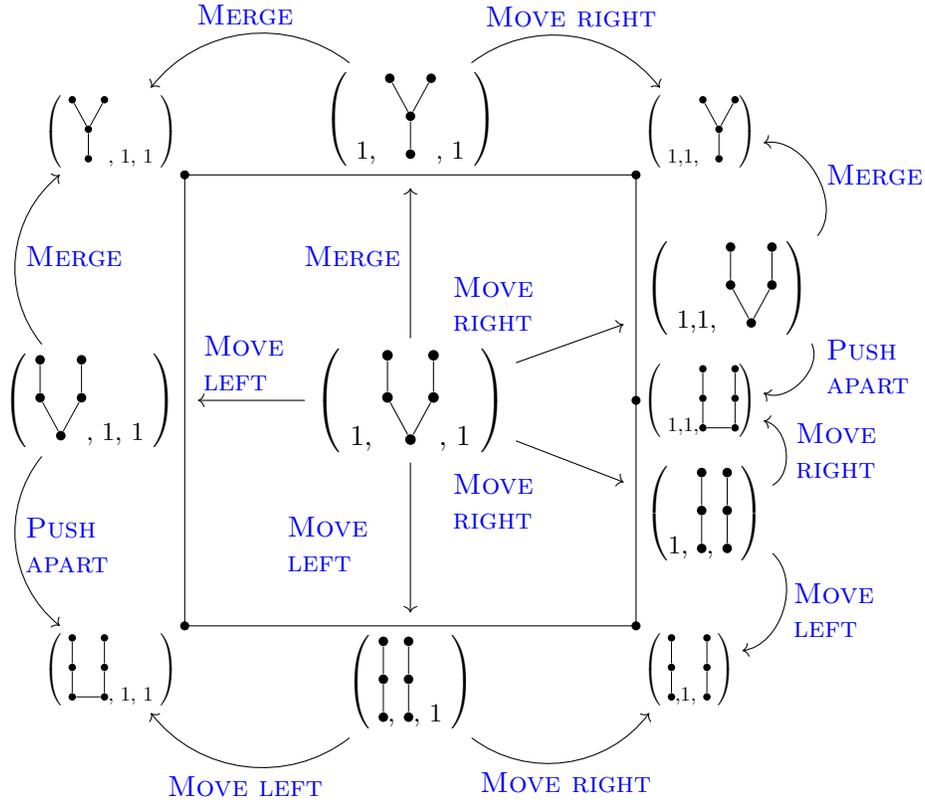
\begin{center}
\begin{tikzpicture}
\draw[thin] (0,0) rectangle (6,6);
\filldraw[black] (0,0) circle (1.5pt);
\filldraw[black] (0,6) circle (1.5pt);
\filldraw[black] (6,0) circle (1.5pt);
\filldraw[black] (6,6) circle (1.5pt);
\filldraw[black] (6,3) circle (1.5pt);

\node[scale = 1] (A) at (3,3) {$\left ( \vcenter{\hbox{ 1,
\begin{forest}
for tree = {grow'=90,circle, fill, minimum width = 4pt, inner sep = 0pt, s sep = 13pt, l sep = 0, l-=5mm}
[{},phantom
[{},  for tree = {fill = black, edge = {color = black}} [ []] [[]]  ]
]
\end{forest}, 1  }} \right )$};

\node[scale = 0.9, anchor = east] (1) at (0,3) {$\left ( \vcenter{\hbox{ \begin{forest}
for tree = {grow'=90,circle, fill, minimum width = 4pt, inner sep = 0pt, l sep = 0, s sep = 13pt, l-=5mm}
[{},phantom
[{},name = one [[]][[]] ]
]
\end{forest},
1, 1  }} \right )$};

\node[scale = 0.9, anchor = north] (2) at (3,0) {$\left ( \vcenter{\hbox{ \begin{forest}
for tree = {grow'=90,circle, fill, minimum width = 4pt, inner sep = 0pt, l sep = 0, s sep = 13pt, l-=5mm}
[{},phantom
[{},name = one [[]] ]
]
\end{forest},
\begin{forest}
for tree = {grow'=90,circle, fill, minimum width = 4pt, inner sep = 0pt, l sep = 0, s sep = 13pt, l-=5mm}
[{},phantom
[{},name = one [[]] ]
]
\end{forest}, 1  }} \right )$};
\node[scale = 0.9, anchor = south] (3) at (3,6) {$\left ( \vcenter{\hbox{ 1,  \begin{forest}
for tree = {grow'=90,circle, fill, minimum width = 4pt, inner sep = 0pt, l sep = 0, s sep = 13pt, l-=5mm}
[{},phantom
[{},name = one [[][]] ]
]
]
\end{forest},
 1  }} \right )$};
\node[scale = 0.9, anchor = west] (4) at (6,1.5) {$\left ( \vcenter{\hbox{1, \begin{forest}
for tree = {grow'=90,circle, fill, minimum width = 4pt, inner sep = 0pt, l sep = 0, s sep = 13pt, l-=5mm}
[{},phantom
[{},name = one [[]] ]
]
\end{forest},
\begin{forest}
for tree = {grow'=90,circle, fill, minimum width = 4pt, inner sep = 0pt, l sep = 0, s sep = 13pt, l-=5mm}
[{},phantom
[{},name = one [[]] ]
]
\end{forest}  }} \right )$};
\node[scale = 0.9, anchor = west] (5) at (6,4.5) {$\left ( \vcenter{\hbox{ 1,1,
\begin{forest}
for tree = {grow'=90,circle, fill, minimum width = 4pt, inner sep = 0pt, s sep = 13pt, l sep = 0, l-=5mm}
[{},phantom
[{},  for tree = {fill = black, edge = {color = black}} [ []] [[]]  ]
]
\end{forest}  }} \right )$};

\node[anchor = north east, scale = 0.7] (11) at (0,0) {$\left ( \vcenter{\hbox{ \begin{forest}
for tree = {grow'=90,circle, fill, minimum width = 4pt, inner sep = 0pt, l sep = 0, s sep = 13pt, l-=5mm}
[{},phantom
[{},name = one [[]] ]
[{},name = two [[]] ]
]
]
\draw[black] (one) -- (two);
\end{forest},
1, 1  }} \right )$};
\node[anchor = north west, scale = 0.7] (22) at (6,0) {$\left ( \vcenter{\hbox{ \begin{forest}
for tree = {grow'=90,circle, fill, minimum width = 4pt, inner sep = 0pt, l sep = 0, s sep = 13pt, l-=5mm}
[{},phantom
[{},name = one [[]] ]
]
\end{forest},1,
\begin{forest}
for tree = {grow'=90,circle, fill, minimum width = 4pt, inner sep = 0pt, l sep = 0, s sep = 13pt, l-=5mm}
[{},phantom
[{},name = one [[]] ]
]
\end{forest} }} \right )$};
\node[anchor = west,scale = 0.7] (33) at (6,3) {$\left ( \vcenter{\hbox{ 1,1,\begin{forest}
for tree = {grow'=90,circle, fill, minimum width = 4pt, inner sep = 0pt, l sep = 0, s sep = 13pt, l-=5mm}
[{},phantom
[{},name = one [[]] ]
[{},name = two [[]] ]
]
]
\draw[black] (one) -- (two);
\end{forest}}} \right )$};
\node[anchor = south west, scale = 0.7] (44) at (6,6) {$\left ( \vcenter{\hbox{ 1,1,\begin{forest}
for tree = {grow'=90,circle, fill, minimum width = 4pt, inner sep = 0pt, l sep = 0, s sep = 13pt, l-=5mm}
[{},phantom
[{},name = one [[][]] ]
]
]
\end{forest}}} \right )$};
\node[anchor = south east, scale = 0.7] (55) at (0,6) {$\left ( \vcenter{\hbox{ \begin{forest}
for tree = {grow'=90,circle, fill, minimum width = 4pt, inner sep = 0pt, l sep = 0, s sep = 13pt, l-=5mm}
[{},phantom
[{},name = one [[][]] ]
]
]
\end{forest},
1, 1  }} \right )$};

\draw[->, shorten > = 5pt] (A) -- node[above,xshift = 0.2cm,blue, text width = 1.5cm]{\sc Move left} (1);
\draw[->, shorten > = 5pt] (A) -- node[left,blue,text width = 1.5 cm]{\sc Move left}(2);
\draw[->, shorten > = 5pt] (A) -- node[left,blue]{\sc Merge}(3);
\draw[->, shorten > = 5pt] (A) -- node[below left,blue, text width = 1.5cm]{\sc Move right} (4);
\draw[->, shorten > = 5pt] (A) -- node[above left,blue, text width = 1.5cm]{\sc Move right} (5);

\draw[->] (1) edge [bend right=45] node[right,blue,text width=1cm]{\sc Push apart} (11);
\draw[->] (5) edge [bend left=70] node[right,xshift = 0.2cm,yshift = 0.2cm,blue,text width = 1.5cm]{\sc Push apart} (33);
\draw[->] (4) edge [bend right=70] node[right,blue, text width = 1.5cm]{\sc Move right} (33);
\draw[->] (5) edge [bend right=70] node[right,blue]{\sc Merge} (44);

\draw[->] (3) edge [bend left=45] node[above,blue]{\sc Move right} (44);
\draw[->] (4) edge [bend left=65] node[right,blue,text width = 1.5cm]{\sc Move left} (22);

\draw[->] (3) edge [bend right=45] node[above,blue]{\sc Merge} (55);
\draw[->] (1) edge [bend left=45] node[right,blue]{\sc Merge} (55);

\draw[->] (2) edge [bend right=45] node[below,blue]{\sc Move right} (22);
\draw[->] (2) edge [bend left=45] node[below,blue]{\sc Move left} (11);

\end{tikzpicture}
\end{center}

A forest-tree-forest triple of dimension 0 has an empty tree in the middle, and in the forests every tree has one branch. We now explain how to compute the coordinates of the vertex labelled with such a triple. 

\begin{prop}
\label{vertcoord}
Let $(\mathfrak{F},1,\mathfrak{G})$ be a such a triple. Mark the leaves {\em right to left} by $L_1$ to $L_n$. Then the coordinates of the corresponding vertex are:

$$v_1 = \begin{cases} 2 & \text{if right forest is nonempty} \\
0 & \text{if right forest is empty}
\end{cases}$$
 and for $i>1$
$$ v_i = \begin{cases} 2 & \text{if } L_{i-1}\text{ and }L_{i} \text{ are on the same tree} \\
1 & \text{if } L_{i-1}\text{ and }L_{i} \text{ are on different trees in the right forest} \\
0 & \text{if } L_{i-1} \text{ and }L_{i} \text{ are on different trees in the left forest/different forests} \end{cases} $$

\end{prop}
\begin{proof}
The result follows by combining main isomorphism in Section 5 in \cite{Pol} (relating forest-tree-forest to ``nice expressions") and the discussion below Definition 4.6 in \cite{Pol} (relating ``nice expressions" to faces in the subdivided cube). We leave the details to an interested reader, since our current goal is to completely avoid ``nice expressions".
\end{proof}

 
As a consequence of the proposition above, the set of $0$-dimensional forest-tree-forest triples inherits a partial order from the cube: corresponding vertices are compared coordinatewise.\\

We conclude by describing minimal and maximal vertex of an arbitrary face.

\begin{prop}
\label{minmax}
Let $(\mathfrak{F},\mathfrak{T},\mathfrak{G})$ be a face. 
\begin{itemize}
    \item Its minimal vertex $V_{\min}$ is obtained by moving $\mathfrak{T}$ to the left and pushing apart at all mid-branch spaces.
    \item Its maximal vertex $V_{\max}$ is obtained by moving $\mathfrak{T}$ to the right and merging at all mid-branch spaces.
\end{itemize}

\begin{proof}
Let $V$ be an arbitrary vertex of $(\mathfrak{F},\mathfrak{T},\mathfrak{G})$. We use the vertex coordinates from Proposition \ref{vertcoord} and show that for every $i$,
$$v_i(V_{\min}) \leq v_i(V) \leq v_i(V_{\max})$$ 
First consider the coordinate $v_1$. If $\mathfrak{G}$ is nonempty, then $v_1 = 2$ for any vertex of $(\mathfrak{F},\mathfrak{T},\mathfrak{G})$. If $\mathfrak{G}$ and $\mathfrak{T}$ are both empty, then $v_1 = 0$ for all vertices of $(\mathfrak{F},\mathfrak{T},\mathfrak{G})$. Now if $\mathfrak{G}$ is empty and $\mathfrak{T}$ is not, then $v_1(V_{\min}) = 0$ since the right forest remains empty after moving $\mathfrak{T}$ to the left, and $v_1(V_{\max}) = 2$ since the right forest becomes nonempty after moving $\mathfrak{T}$ to the right.\\

Now consider the coordinate $v_i$ for $i>1$. We go through the possibilities for the leaves $L_{i-1}$ and $L_{i}$ in $(\mathfrak{F},\mathfrak{T},\mathfrak{G})$.

\begin{itemize}
    \item $L_{i-1}$ and $L_{i}$ are on the same branch (does not matter of which tree). Then 
    $$2 = v_i(V_{\min}) = v_i(V) = v_i(V_{\max})$$ 
    \item $L_{i-1}$ and $L_{i}$ are on different branches of the same tree in $\mathfrak{F}$. Then pushing these branches apart forces the value $0$, and merging these branches forces the value $2$, thus
    $$0 = v_i(V_{\min}) \leq  v_i(V) \leq v_i(V_{\max}) = 2$$ 
    \item $L_{i-1}$ and $L_{i}$ are are on different branches of the same tree in $\mathfrak{T}$. Then moving $\mathfrak{T}$ to the left and pushing these branches apart forces the value $0$, and moving $\mathfrak{T}$ to the right and merging these branches forces the value $2$, thus
    $$0 = v_i(V_{\min}) \leq  v_i(V) \leq v_i(V_{\max}) = 2$$ 
    \item $L_{i-1}$ and $L_{i}$ are are on different branches of the same tree in $\mathfrak{G}$. Then pushing these branches apart forces the value $1$, and merging these branches forces the value $2$, and the value $0$ (corresponding to different trees in left/different forests) is impossible, thus
    $$1 = v_i(V_{\min}) \leq  v_i(V) \leq v_i(V_{\max}) = 2$$ 
    \item $L_{i-1}$ and $L_{i}$ are on different trees in $\mathfrak{F}$. Then 
    $$0 = v_i(V_{\min}) =  v_i(V) = v_i(V_{\max}) $$ 
    \item $L_{i-1}$ is in $\mathfrak{F}$ and $L_i$ is in $\mathfrak{T}$. Then moving $\mathfrak{T}$ to the left makes them belong to different trees of the left forest, and forcing $\mathfrak{T}$ to the right makes them belong to different trees in different forests, anyways forcing the value 0:
    $$0 = v_i(V_{\min}) =  v_i(V) = v_i(V_{\max}) $$ 
    \item $L_{i-1}$ is in $\mathfrak{F}$ and $L_i$ is in $G$ (then $\mathfrak{T}$ has to be empty), then
    $$0 = v_i(V_{\min}) =  v_i(V) = v_i(V_{\max}) $$ 
    \item $L_{i-1}$ is in $\mathfrak{T}$ and $L_i$ is in $\mathfrak{G}$. Then moving $\mathfrak{T}$ to the left forces the value $0$ and moving $\mathfrak{T}$ to the right forces the value $1$, and the value 0 is impossible, thus
    $$0 = v_i(V_{\min}) \leq  v_i(V) \leq v_i(V_{\max}) = 1 $$ 
    \item $L_{i-1}$ and $L_i$ are on different trees in $\mathfrak{G}$.
    $$0 = v_i(V_{\min}) =  v_i(V) = v_i(V_{\max}) $$ 
\end{itemize}
\end{proof}

\end{prop}

\section{Shortness}
We now explain {\em shortness}  --  a combinatorial property that is later proved to be enjoyed by freehedra. Shortness implies that a certain colored DG-operad associated to a polytope is augmented -- and, conjecturally, self-dual. \\

As explained above, freehedra come with a natural partial order on their vertices. This is axiomatized in the definition of a directed polytope (\cite{AP}).

\begin{defi}
A polytope $P$ is {\it directed} if its 1-skeleton is an oriented graph with no cycles, one source and one sink, and the same holds for every face of $P$.
\end{defi}

Edge directions make vertices of $P$ into a partially ordered set, and this extends to a relation on faces.

\begin{defi}
Let $F_1$ and $F_2$ be two faces of a directed polytope $P$. Then $F_1 \leq F_2$ if $\min F_1 \leq \max F_2$.
\end{defi}

This allows speaking of face chains, and their {\it excess}.

\begin{defi} A sequence of faces $(F_1, \ldots, F_n)$ is a {\it face chain} in a face $F$ if $F_i \subset F$ for any $i$ and $F_1 \leq \ldots \leq F_n$. The {\it excess} of $(F_1, \ldots, F_n)$ in $F$ is $(\dim F-1) - \sum (\dim F_i-1)$. The set of face chains in $F$ of length $n$ with excess $l$ is denoted by $\mathfrak{fc}_l(F,n)$. \end{defi}

We are now in position to define shortness. \\

\begin{defi}
A directed polytope $P$ is called short if nontrivial chains have positive excesses.
\end{defi}

Examples of short polytopes include simplices and cubes with their standard directions. Nonexamples include associahedra with Tamari directions and permutahedra with weak Bruhat directions.\\

To an directed polytope, no matter short or not, one associates a colored DG-operad via the following construction.

\begin{defi}
The set of colors is given by all faces of $P$. The operation spaces are 

$$O_P(F_1, \ldots, F_n; F) = \begin{cases} k[l-n+1] & (F_1, \ldots, F_n) \in \mathfrak{fc}_l(F,n) \\ 0 & \text{else} \end{cases} $$

Thus the total grading of every operation is the excess of the corresponding chain. The composition maps are either $k[m] \simeq k[m]$ or $0  \to k[m]$.
\end{defi}

The correctness of the above definition is proved in \cite{AP}, Theorem 2.6. \\

In terms of $O_P$, shortness means that $O_P$ is augmented: with respect to the total grading, its degree zero part consists of identity operations, and all other operations are of higher degrees. \\

The reason to be interested in short polytopes is the following conjecture from \cite{AP}.

\begin{con}
\label{conj}
For short polytopes their operads are Koszul and Koszul self-dual.
\end{con}

We are mainly interested in a certain numeric consequence of self-duality: the involutive property of Poincare-Hilbert series. For a colored DG-operad $\mathcal{P}$ let $T(\mathcal{P})$ be the algebra of formal power series in non-commuting variables $c \in \operatorname{Colors}(\mathcal{P})$, with coefficients in $k((t))$ (this means that the extra invertible variable $t$ is required to commute with everything).  For $V = \bigoplus V_i$ a graded vector space, let $\dim V$ be the Laurent series in $t$ whose coefficient near $t^n$ is the dimension of $V_n$. Let $\operatorname{Tuples}(\mathcal{\mathcal{P}})$ be the set of all ordered tuples of elements in $\operatorname{Colors}(\mathcal{P})$ (repetitions are allowed) - these are possible inputs of operations. For $\overset{\to}{u} \in \operatorname{Tuples}(\mathcal{P})$, let $|\overset{\to}{u}|$ denote the length of the tuple. Notice that any such tuple can be viewed as a non-commutative monomial in $T(\mathcal{P})$.

\begin{defi}
The Poincar\'e-Hilbert endomorphism for $\mathcal{P}$ is an endomorphism $f_\mathcal{P}$ of $T(\mathcal{P})$, which is defined on generators by 
$$f_\mathcal{P}(c) = \sum_{ \overset{\to}{u} \in \operatorname{Tuples}(\mathcal{P})} \left(t^{|\overset{\to}{u}|-1}\right) \left[\dim \mathcal{P}(\overset{\to}{u};c) \right] \overset{\to}{u} $$
and extended to an algebra map by multiplication since the algebra in question is free.
\end{defi}

\begin{fact}
\label{f}
Let $I$ be an endomorphism of $T(\mathcal{P})$ sending each variable $c$ to $-c$, and $t$ to $-t$. For a colored operad $\mathcal{P}$ and its Bar-dual cooperad $\operatorname{Bar}(\mathcal{P})$, we have 
$ f_\mathcal{P} \circ I \circ f_{\operatorname{Bar}(\mathcal{P})} \circ I = \operatorname{Id}$.
\end{fact}

For $\mathcal{P}$ being $O_P$ as above for a short $P$, Conjecture \ref{conj} and Fact \ref{f} imply that teh cellular complex of $P$ is an integrated $A_\infty$-coalgebra in the sense of the following definition:

\begin{defi}
The structure of an integrated $A_\infty$-coalgebra
 on a graded vector space $V$ is the data of $t$-linear involution on the completed tensor algebra $\widehat{T}(V[1])$ 
\end{defi}

\section{Proof of shortness}
We now prove the main result of the present paper.

\begin{theo}
Freehedra are short. 
\end{theo}

To prove it, we introduce the following technical notion.

\begin{defi}
A $\mathbb{N}$-valued function $D$ on the vertices of $P$ is called {\em sup-dimensional}, if is satisfies the following conditions:
\begin{itemize}
    \item $D(\min P) = \dim P -1$ and $D(\max P) = 0$
    \item for every face $F \subset P$, $D(\min P) - D(\max P) \geq \dim F - 1$
\end{itemize}
\end{defi}

In the example of the standard simplex, the sup-dimensional function is simply the number of the vertex. Indeed, an $m$-dimensional facet has the numbers of its minimum and maximum vertices differ at least by $m$; sometimes by a number a lot larger than $m$, e.g. consider the edge connecting $0$ to $n$. \\

The reason to consider sup-dimensional functions is the following lemma.

\begin{lem}
For any ordered $P$ polytope equipped with a sup-dimensional function $D$, the excess of nontrivial chains within top-dimensional face is nonnegative.
\end{lem}

\begin{proof}
We first observe that shortness can be only checked for {\em connected} chains $(F_1, \ldots, F_n)$ with $\max F_i = \min F_{i+1}$. Indeed, an arbitrary chain can be extended to a connected chain by adding the edge-path between $\max F_i$ and $\min F_{i+1}$, without affecting its excess. Let $(F_1, \ldots, F_n)$ be such a chain. We name the vertices: set $v_0 = \min F_1$, $v_i = \max F_i = \min F_{i+1}$ for $0 < i < n$, and $v_n = \max F_n$. By assumption of $D$ being sup-dimensional, we have 
$$ D(v_{i}) - D(v_{i-1}) \geq \dim F_i - 1 $$
Adding this up for all values of $i$, we get 
$$ D(v_n) - D(v_0) \geq \sum_{i=1}^n (\dim F_i-1)$$
but $D(v_n) - D(v_0) = \dim P - 1$, thus proving the statement.
\end{proof}

Note that the Lemma does not imply as much as shortness of $P$. Indeed, shortness has to be checked within all faces not just the top-dimensional one. Additionally, the Lemma does not outrule chains of zero excess, which are not allowed in a short polytope. To outrule them, one needs to prove that every connected face chain from $\min P$ to $\max P$ includes at least one face $F$ for which the inequality $D(\min P) - D(\max P) \geq \dim F - 1$ is strict. \\

We now construct a sup-dimensional function for freehedra. \\

\begin{defi}
Let $v$ be a vertex in the freehedron $F_{n}$, given by a tree-forest-tree triple $(\mathfrak{F},1,\mathfrak{G})$. Define 
$$ D(v) = a+b-1$$
where $a$ is the number of trees in $\mathfrak{F}$ and $b$ is the number of trees in $\mathfrak{G}$.
\end{defi}

\begin{lem}
The function $D$ defined above is sup-dimensional.
\end{lem}

\begin{proof}
For $\min F_n$ we have $n$ one-leave trees in the left forest and no trees in the right forest, so $D(\min F_n) = n-1 = \dim F_n - 1$. For $\max F_n$, we have no trees in the left forest and one tree with $n$ leaves in the right forest, so $D( \max F_n) = 1-1 = 0$. \\

Now let $F$ be an arbitrary face with label $(\mathfrak{F},\mathfrak{T},\mathfrak{G})$. Suppose there are $k$ trees and $d$ mid-branch spaces. If $\mathfrak{T}$ is empty, $\dim F $ is exactly $d$, and if $T$ is nonempty, $\dim F$ is $d+1$. Its minimal vertex and maximal vertex are described in Proposition \ref{minmax}. Passage to minimal vertex requires moving $\mathfrak{T}$ to the right and pushing apart at all mid-branch spaces. Each instance of pushing apart creates a new tree, thus $D( \min F) = k+d$. Passage to maximal vertex requires moving $\mathfrak{T}$ to the right and merging at all mid-branch spaces. Merging neither produces new trees nor unites the old ones, so $D (\max F) = k$. Thus $D (\min F) - D(\max F) = k+d-k = d$, which is indeed greater or equal to $\dim F-1$.
\end{proof}

Note that we get a simple criterion for when the inequality $D(\min F) - D(\max F) \geq \dim F - 1$ is strict: this happens when the forest-tree-forest triple for $\mathfrak{F}$ has a nonempty middle tree. In this case sides differ by 1.\\

We are now left to outrule chains of zero excess, and to treat the faces which are not top-dimensional.

\begin{lem}
Every connected face chain from $\min \F_n$ to $\max \F_n$ includes at least one face $F$ for which the inequality $D(\min F) - D(\max F) \geq \dim F - 1$ is strict.
\end{lem}

\begin{proof}
Let $(F_i)$ be a connected face chain, and $(v_i)$ its vertex sequence. In the forest-tree-forest triple for $v_0 = \min F$ all the leaves are in the left forest, and in the forest-tree-forest triple for $v_n = \max F$ all the leaves are in the right forest. Consequently, there exists such an index $i$ that the left forest for $v_i$ has {\em less} leaves then the left forest for $v_{i+1}$, say by $l$ leaves less. Then the face $F_i$ with $\min F_i = v_i$ and $\max F_i = v_{i+1}$ has $l$ leaves in the middle tree, thus satisfies the criterion.
\end{proof}

\begin{cor}
Within the top-dimensional face of $\F_n$, there are no nontrivial chains of zero excess.
\end{cor}

Finally we recall that every face of a freehedron is a product of smaller-dimensional freehedra and cubes, compatibly with directions. A face chain in such a product of polytopes obviously corresponds to a product of chains, with its excess being the sum of respective excesses.

\begin{cor}
Within any face, nontrivial chains have positive excess.
\end{cor}

This finishes the proof of freehedra's shortness.

\end{document}